\documentclass[11pt]{article}
\usepackage{amsmath, amssymb, amsthm,enumerate}
\usepackage[]{todonotes}
\usepackage{dsfont}
\usepackage{hyperref}
\hypersetup{
	colorlinks=true,
    linkcolor={red!50!black},
    citecolor={blue!50!black},
    bookmarksopen=true,
	  bookmarksnumbered,
	  bookmarksopenlevel=2,
	  bookmarksdepth=3
}
\usepackage[noabbrev,capitalize]{cleveref}

\usepackage{color}
\usepackage{tikz}
\usetikzlibrary{calc}
\usetikzlibrary{math}
\usetikzlibrary{patterns}
\usetikzlibrary{decorations.pathreplacing}
\usepackage{ifthen}
\usepackage{amssymb}
\usepackage{amsfonts}
\usepackage{mathtools}
\usepackage[mathscr]{euscript}
\usepackage{adjustbox}
\allowdisplaybreaks
\expandafter\let\expandafter\oldproof\csname\string\proof\endcsname
\let\oldendproof\endproof
\renewenvironment{proof}[1][\proofname]{%
	\oldproof[\bf #1]%
}{\oldendproof}

\allowdisplaybreaks
\newcommand{\PP}{\textsf{P}}
\newcommand{\NP}{\textsf{NP}}
\newcommand{\NPc}{\textsf{NPc}}
\theoremstyle{plain}

\makeatletter
\newtheorem*{rep@theorem}{\rep@title}
\newcommand{\newreptheorem}[2]{%
\newenvironment{rep#1}[1]{%
 \def\rep@title{#2 \ref{##1}}%
 \begin{rep@theorem}}%
 {\end{rep@theorem}}}
\makeatother

\newreptheorem{theorem}{Theorem}
\newreptheorem{lemma}{Lemma}

\newtheorem{lemma}{Lemma}[section]
\newtheorem{theorem}[lemma]{Theorem}
\newtheorem{claim}{Claim}
\newtheorem{proposition}[lemma]{Proposition}
\newtheorem{observation}[lemma]{Observation}
\newtheorem{corollary}[lemma]{Corollary}

\newtheorem{definition}[lemma]{Definition}

\theoremstyle{definition}
\newtheorem{remark}[lemma]{Remark}

\RequirePackage[normalem]{ulem} 
\RequirePackage{color}\definecolor{RED}{rgb}{1,0,0}\definecolor{blue}{rgb}{0,0,1} 

\title{
\vspace{-0.8cm}
On balanceable and simply balanceable regular graphs}

\author{
Milad Ahanjideh\\
\small FAMNIT and IAM, University of Primorska, Koper, Slovenia\\
\small \texttt{milad.ahanjideh@upr.si}
\and
Martin Milani\v c\\
\small FAMNIT and IAM, University of Primorska, Koper, Slovenia\\
\small \texttt{martin.milanic@upr.si}
\and
Mary Servatius\\
\small Koper, Slovenia\\
\small \texttt{emservat@gmail.com}
}
\date{}

\begin{document}

    \maketitle

    \begin{abstract}
    We continue the study of balanceable graphs, defined by Caro, Hansberg, and Montejano in 2021 as graphs $G$ such that any $2$-coloring of the edges of a sufficiently large complete graph containing sufficiently many edges of each color contains a balanced copy of $G$.
    While the problem of recognizing balanceable graphs was conjectured to be $\NP$-complete by  Dailly, Hansberg, and Ventura in 2021, balanceable graphs admit an elegant combinatorial characterization: a graph is balanceable if and only there exist two vertex subsets, one containing half of all the graph's edges and another one such that the corresponding cut contains half of all the graph's edges.
    We consider a special case of this property, namely when one of the two sets is a vertex cover, and call the corresponding graphs simply balanceable.
    We prove a number of results on balanceable and simply balanceable regular graphs.
    First, we characterize simply balanceable regular graphs via a condition involving the independence number of the graph.
    Second, we address a question of Dailly, Hansberg, and Ventura from 2021 and show that every cubic graph is balanceable.
    Third, using Brooks' theorem, we show that every $4$-regular graph with order divisible by $4$ is balanceable.
    Finally, we show that it is \NP-complete to determine if a $9$-regular graph is simply balanceable.\\
    \noindent


    \end{abstract}

    \section{Introduction}

\subsection{Background: balanceable graphs}

A \emph{$2$-coloring} of the edges of a complete graph $K_n$ is a function $\varphi$ from the edge set of $K_n$ to the set $\{R, B\}$ of two colors, red and blue.
  The well-known Ramsey's theorem (see~\cite{MR1576401}) asserts that if $G$ is a complete graph, then any $2$-coloring of the edges of a large enough complete graph $K_n$ yields a monochromatic copy of $G$.
  Caro et al.~\cite{Caro2} (see also~\cite{Caro1}) considered a different approach to $2$-colorings of $K_n$, asking for a \emph{balanced copy} of a graph $G$ in such a coloring and defined the notion of \emph{balanceability}.

  A $2$-coloring $\varphi$ of the edges of $K_n$ is said to contain a balanced copy of a graph $G$ if there exists a copy of $G$ in $K_n$ such that its edge set $E$ can be partitioned in two parts $\left(E_1, E_2\right)$ such that $\left\|E_1|-|E_2\right\| \leq 1$ and $\varphi(e)=R$ for all $e \in E_1$, and $\varphi(e)=B$ for all $e \in E_2$.

  Let $G$ be a graph and $n$ be a positive integer. Let {\rm bal}$(n, G)< \lfloor\frac{1}{2}{n\choose 2}\rfloor$ be the smallest integer, if it exists, such that
  every $2$-coloring $\varphi\colon E\left(K_n\right) \rightarrow\{R, B\}$ of the edges of the complete graph $K_n$ with $\left|\varphi^{-1}(R)\right|,\left|\varphi^{-1}(B)\right|>$ {\rm bal}$(n, G)$ contains a balanced copy of $G$.\footnote{In~\cite{Caro2}, the condition {\rm bal}$(n, G)< \lfloor\frac{1}{2}{n\choose 2}\rfloor$ was not stated; however, this condition is obviously necessary, since without it, {\rm bal}$(n, G)$ would exist for every $n$, for a trivial reason.}
If there exists an integer $n_0$ such that, for every integer $n \geq n_0$, {\rm bal}$(n, G)$ exists, then $G$ is said to be \emph{balanceable}.

  It seems challenging to decide whether a graph is balanceable.
  Caro, Hansberg, and Montejano~\cite{Caro2} gave a characterization of balanceable graphs (see \cref{sec:SBG}) and used it to show that all trees are balanceable.
  They also introduced a new family of balanceable graphs called amoebas and developed their study further in~\cite{Caro6}.
  In an earlier work~\cite{Caro1}, the same authors showed that $K_4$ is the only balanceable complete graph with an even number of edges.
  In~\cite{main}, Dailly, Hansberg, and Ventura gave some conditions for a graph to be balanceable and characterized the balanceability property within rectangular and triangular grids, as well as for special circulant graphs.
  In general, however, the question of which graphs are balanceable is widely open.
  In fact, Dailly, Hansberg, and Ventura conjectured in~\cite{main} that the problem of determining whether a given graph is balanceable is \NP-complete.
Further interesting results on the notion of balanceability of graphs, variants, and generalizations can be found in~\cite{Caro3,Caro5,Dail}.

\subsection{Simply balanceable graphs}\label{sec:SBG}

Caro, Hansberg, and Montejano~\cite{Caro2} gave a useful characterization of balanceability.
To state it, we need a couple of definitions.
Given two integers $k$ and $\ell$, we say that $\ell$ is \emph{essentially half} of $k$ if $\ell \in\left\{\left\lfloor\frac{k}{2}\right\rfloor,\left\lceil\frac{k}{2}\right\rceil\right\}$.
The \emph{size} of a graph $G$ is the number of edges in $G$.
Further, given a graph $G = (V,E)$, we say that a set $X\subseteq V$ is an \emph{internally balanced set} in $G$ if the number of edges with both endpoints in $X$ is essentially half of the size of $G$.
Similarly, $X$ is an \emph{externally balanced set} in $G$ if the number of edges with exactly one endpoint in $X$ is essentially half of the size of $G$.
In these terms, the announced characterization of balanceability due to Caro, Hansberg, and Montejano (see~\cite[Theorem 2.6]{Caro2}) can be stated as follows.

\begin{theorem}\label{thm:CHM}
A graph $G$ is balanceable if and only if there exists an internally balanced set in $G$ as well as an externally balanced set in $G$.
\end{theorem}

An interesting special case occurs if a graph $G$ admits an internally (or externally) balanced set that is also a \emph{vertex cover}, that is, a set of vertices touching all edges.
Indeed, if $C$ is a vertex cover in $G$, then every edge in $G$ has either both endpoints in $C$ or exactly one endpoint in $C$.
This immediately implies that the following conditions are equivalent: (i) $C$ is internally balanced, (ii) $C$ is externally balanced, and (iii) the complementary set $V\setminus C$ is externally balanced (the equivalence between (ii) and (iii) holds for any set $C\subseteq V$).
Furthermore, since a set $X\subseteq V$ is a vertex cover if and only if $V\setminus X$ is an independent set, the following conditions are equivalent for a graph $G$:
\begin{enumerate}[(i)]
\item there exists an internally balanced vertex cover in $G$,
\item there exists an externally balanced vertex cover in $G$, and
\item there exists an externally balanced independent set in $G$.
\end{enumerate}
We say that a graph is \emph{simply balanceable} if it satisfies one of the above three equivalent conditions (and, hence, all of them).
The equivalence of these three conditions and \cref{thm:CHM} imply that every simply balanceable graph is balanceable, which justifies the choice of the name.

Many arguments establishing balanceability of certain graphs (see, e.g., Dailly, Hansberg, and Ventura~\cite{main}) are based on showing the existence of an externally balanced independent set, that is, on showing that the graph is in fact simply balanceable.

\subsection{Our results}

In this paper we continue the study of balanceable and simply balanceable graphs, with a focus on regular graphs.
We prove four main results.

Our first main result is the following characterization of simply balanceable regular graphs via a condition involving the independence number of the graph.

\begin{reptheorem}{lem:regular-sb}
Let $k\ge 2$ be an integer and let $G$ be an $n$-vertex $k$-regular graph.
Then $G$ is simply balanceable if and only if one of the following conditions holds.
\begin{enumerate}
\item $n\not\equiv 2~({\rm mod}~4)$ and $k = 2$.
\item $n\equiv 0~({\rm mod}~4)$ and $\alpha(G)\ge n/4$,
\end{enumerate}
\end{reptheorem}

\Cref{lem:regular-sb} has several consequences.
First, it shows that in the case of $2$-regular graphs, the balanceability and simple balanceability conditions are equivalent.
Second, it implies that if $G$ is a cubic graph of order $n$, then $G$ is simply balanceable if and only if $n\equiv 0~({\rm mod}~4)$.
Furthermore, combined with Brooks' theorem, it leads to the same characterization of simple balanceability for connected $4$-regular graphs.

In contrast to the challenging question on recognizing balanceable graphs, our second main result answers the analogous question for simply balanceable graphs.
The proof of this result is facilitated by \Cref{lem:regular-sb}.

\begin{reptheorem}{NP-complete}
Determining whether a given $9$-regular graph is simply balanceable is \NP-complete.
\end{reptheorem}

Next, we address a question of Dailly, Hansberg, and Ventura regarding balanceability of cubic graphs (see~\cite{main}).

\begin{reptheorem}{Thm:cubic graph}
Every cubic graph is balanceable.
\end{reptheorem}

Finally, we turn our attention to the case of $4$-regular graphs.
In this case, a necessary condition for balanceability is that the graph has even order.
In the case of order $n\equiv 2~({\rm mod}~4)$, some graphs are balanceable and some are not (see~\cite{main}).
In contrast to this, we show, using Brooks' theorem, that there are no exceptions to balanceability in the case of order $n\equiv 0~({\rm mod}~4)$.

\begin{reptheorem}{Thm:4-regular}
Every $4$-regular graph of order $n$, where $n\equiv 0~({\rm mod}~4)$, is balanceable.
\end{reptheorem}

\Cref{Thm:4-regular} generalizes Lemmas 15 and 16 in~\cite{main}, which establish that certain $4$-regular circulant graphs are balanceable.

\paragraph{Structure of the paper.}
After collecting the necessary preliminaries in \Cref{sec:prelim}, we
characterize simply balanceable regular graphs (\Cref{lem:regular-sb}) and derive several corollaries in \Cref{sec:regular}.
We prove the \NP-completeness result (\cref{NP-complete}) in \cref{sec:complexity}.
Questions regarding balanceability of cubic and $4$-regular graphs are investigated in \cref{cubic,4-regular}, respectively.
We conclude the paper in \Cref{sec:concl} with a summary of our results and some open questions.

\section{Preliminaries}\label{sec:prelim}

Throughout this paper, all graphs are assumed to be finite, simple (that is, without loops or multiple edges), and undirected.
The vertex and the edge set of a graph $G$ are denoted with $V(G)$ and $E(G)$, respectively.
The \emph{order} of $G$ is the number of vertices.
For a vertex $v\in V(G)$, we denote by $N_G(v)$ the set of all neighbors of $v$ and by $d_G(v)$ (or simply by $d(v)$ if the graph is clear from the context) its degree.
For a set $X\subseteq V(G)$, we denote by $N(X)$ the set of vertices in $V(G)\setminus X$ that have a neighbor in $X$ and by $N[X]$ the set $N(X)\cup X$.
The maximum degree of a vertex in $G$ is denoted by $\Delta(G)$.
For a nonnegative integer $k$, a graph is said to be \emph{$k$-regular} if all its vertices have degree $k$.
We will refer to $3$-regular graphs also as \emph{cubic} graphs.
A graph is \emph{regular} if it is $k$-regular for some $k\ge 0$.

Let $G = (V,E)$ be a graph.
A \emph{vertex cover} in $G$ is a set $C\subseteq V$ such that every edge of $G$ has an endpoint in $C$.
An \emph{independent set} in $G$ is a set of pairwise nonadjacent vertices.
We denote by $\alpha(G)$ be the \emph{independence number} of $G$, that is, the maximum cardinality of an independent set in $G$, and by $\chi(G)$ its \emph{chromatic number}, that is, the smallest integer $k$ such that $G$ is $k$-colorable (that is, it admits a partition of the vertex set into $k$ independent sets).
Coloring greedily vertices of a graph one at a time shows that every graph $G$ satisfies $\chi(G)\le \Delta(G)+1$.
Furthermore, Brooks' theorem~\cite{Brook} asserts that every connected graph $G$ that is neither an odd cycle nor a complete graph is $\Delta(G)$-colorable.

A \emph{matching} in a graph is a set of pairwise disjoint edges.
A graph $G$ is \emph{$k$-edge-colorable} if it admits a partition of the edge set into $k$ matchings.
The \emph{chromatic index} of $G$, denoted $\chi'(G)$, is the smallest integer $k$ such that $G$ is $k$-edge-colorable.
The \emph{line graph} of a graph $G= (V,E)$ is the graph $L(G)$ with vertex set $E$ in which distinct two vertices are adjacent if and only if the corresponding edges of $G$ share an endpoint.

The \emph{Cartesian product} of two graphs $G$ and $H$ is the graph $G \Box H$ with vertex set $V(G \Box H)=V(G) \times V(H)$, and for which $(x, u)(y, v)$ is an edge if and only if either $x=y$ and $u v \in E(H)$, or $x y \in E(G)$ and $u=v$. We denote by $G+H$ the disjoint union of graphs $G$ and $H$.
For a graph $G$ and a set $X\subseteq V(G)$, we denote by $G[X]$ the subgraph of $G$ induced by $X$, that is, the graph with vertex set $X$ in which two vertices are adjacent if and only if they are adjacent in $G$.
By $G-X$ we denote the subgraph of $G$ induced by $V(G)\setminus X$.

Next, we recall the original formulation of the characterization of balanceability given by \cref{thm:CHM}.
For a graph $G$ and a set $X\subseteq V(G)$, we denote by $e_G(X)$ the number of edges in the subgraph of $G$ induced by $X$, and, for $Y\subseteq V(G)\setminus X$, by $e_G(X, Y)$ the number of edges with one endpoint in $X$ and the other in $Y$.

\begin{theorem}[A reformulation of \cref{thm:CHM}]\label{partition}
A graph $G = (V,E)$ is balanceable if and only if $G$ has both a set of vertices $W \subseteq V$ and a partition $V=\{X,Y\}$ such that
\begin{equation}\label{eq:balanceable}
e_G(W),\, e_G(X, Y) \in\left\{\left\lfloor\frac{|E|}{2}\right\rfloor,\left\lceil\frac{|E|}{2}\right\rceil\right\}\,.
\end{equation}
\end{theorem}

An interesting special case arises when the conditions from \cref{partition} are satisfied with $W = X$.
In particular, this is the case if $W$ is a vertex cover in $G$ such that
\[\left\lfloor\frac{|E|}{2}\right\rfloor\le e_G(W) \le \left\lceil\frac{|E|}{2}\right\rceil\,.\]
Indeed, in this case, setting $Y = V(G)\setminus W$ we have $e_G(Y) = 0$ and hence $|E| = e_G(W) +e_G(X,Y)$, implying \eqref{eq:balanceable}.
Reformulating the condition on $W$ in terms of its complement $V(G)\setminus W$, which is an independent set, leads to the following definition.

\begin{definition}\label{def:sb}
A graph $G = (V,E)$ is \emph{simply balanceable} if there exists an independent set $I$ in $G$ such that
\[\left\lfloor\frac{|E|}{2}\right\rfloor\le \sum_{x \in I} d(x)\le \left\lceil\frac{|E|}{2}\right\rceil\,.\]
\end{definition}

Note that, in terms of the terminology used in \Cref{sec:SBG}, a graph $G$ is defined to be simply balanceable if there exists an externally balanced independent set in $G$.

\begin{proposition}
\label{indep}
Every simply balanceable graph is balanceable.
\end{proposition}

\begin{proof}
Let $G = (V,E)$ be simply balanceable and let $I$ be an independent set $I$ in $G$ such that $\left\lfloor\frac{|E|}{2}\right\rfloor\le \sum_{x \in I} d(x)\le \left\lceil\frac{|E|}{2}\right\rceil$.
Then, condition \eqref{eq:balanceable} from \Cref{partition} is satisfied by the set $W = V\setminus I$ and the partition $\{X,Y\}$ of $V$ given by $X = I$ and $Y = V\setminus I$ and, hence, $G$ is balanceable.
\end{proof}

\Cref{indep} for even $|E|$ was extensively used in Dailly, Hansberg, and Ventura~\cite{main}, especially for graphs with an even number of edges.
In particular, for bipartite regular graphs, \Cref{indep} implies the following (see~\cite[Proposition 10]{main}).

\begin{sloppypar}
\begin{proposition}\label{bip-regular}
Every regular bipartite graph of order \hbox{$n\equiv 0~({\rm mod}~4)$} is simply balanceable.
\end{proposition}
\end{sloppypar}

It is not difficult to construct examples showing that the class of balanceable graphs is not closed under disjoint union (see~\Cref{remark}).
Nevertheless, such a closure property holds under an additional assumption.

\begin{lemma}\label{lem:join}
    Let $G$ and $H$ be balanceable graphs such that $|E(H)|$ is even.
    Then $G+H$ is balanceable.
\end{lemma}
\begin{proof}
    Assume that $2k\leq |E(G)|\leq 2k+1$ and $|E(H)|=2\ell$ and for two integers $k$ and $\ell$.
    Since $G$ is balanceable, by \Cref{partition}, there exist subsets $W$ and $X$ of $V(G)$ such that $\ell\le e_G(W)\le \ell+1$ and $\ell\le e_G(X, V(G)\setminus X)\le\ell+1$.
    Similarly, there are subsets $Z$ and $Y$ of $V(H)$ such that $e_H(W)=e_H(Y, V(H)\setminus Y)=k$.
    Set $U = W\cup Z$ and $S = X\cup Y$.
    Then \[e_{G+H}(U) = e_G(W)+e_H(Z) \in \{k+\ell,k+\ell+1\}\]
    and, similarly,
    \[e_{G+H}(S,V(G+H)\setminus S) = e_G(X, V(G)\setminus X)+e_H(Y, V(H)\setminus Y) \in \{k+\ell,k+\ell+1\}\,.\]
    Since $2\ell+2k \leq |E(G+H)|\leq 2\ell+2k+1$, it follows that
    \[e_{G+H}(U),\, e_{G+H}(S,V(G+H)\setminus S) \in\left\{\left\lfloor\frac{|E(H+G)|}{2}\right\rfloor,\left\lceil\frac{|E(H+G)|}{2}\right\rceil\right\}\,.\]
Hence $G+H$ is balanceable by \Cref{partition}.
\end{proof}

Let us say that a graph is \emph{even} if all its vertices have even degree.
We will make use of the following reformulation of~\cite[Proposition 7]{main}.

\begin{proposition}[Dailly, Hansberg, and Ventura~\cite{main}]
\label{prop:Eulerian}
Let $G=(V,E)$ be a balanceable even graph.
Then $|E|$ is either odd or a multiple of $4$.
\end{proposition}

\section{A characterization of simply balanceable regular graphs}\label{sec:regular}

In this section, we address the question of which regular graphs are simply balanceable.
We first deal with the following easy special case.

\begin{observation}\label{obs:1-regular}
Every $1$-regular graph is simply balanceable.
\end{observation}

\begin{proof}
Let $G=(V,E)$ be a $1$-regular graph.
Then, $G$ is a disjoint union of $m$ edges, for some positive integer $m$, and any set of $\lceil{m/2}\rceil$ vertices, no two from the same edge, is an externally balanced independent set, that is, an independent set $I$ satisfying the condition of \Cref{def:sb}.
\end{proof}

The remaining cases are characterized as follows.

\begin{theorem}\label{lem:regular-sb}
Let $k\ge 2$ be an integer and let $G$ be an $n$-vertex $k$-regular graph.
Then $G$ is simply balanceable if and only if one of the following conditions holds.
\begin{enumerate}
\item $k = 2$ and $n\not\equiv 2~({\rm mod}~4)$.
\item $k \ge 3$, $n\equiv 0~({\rm mod}~4)$, and $\alpha(G)\ge n/4$.
\end{enumerate}
\end{theorem}

\begin{proof}
We first show that the conditions are sufficient.
Assume first that $k\ge 3$, $n\equiv 0~({\rm mod}~4)$, and $\alpha(G)\ge n/4$.
Then, there exists an independent set $I$ in $G$ such that $|I| = n/4$.
Note that $|E| = kn/2$ and, hence, $\sum_{x \in I} d(x) = k|I| = kn/4 = |E|/2$.
Therefore, $G$ is simply balanceable.

Assume now that $k = 2$ and $n\not\equiv 2~({\rm mod}~4)$.
Since $G$ is $2$-regular, we have $|E| = n$.
Furthermore, $G$ is $3$-colorable and hence $\alpha(G)\geq n/3$.
We consider three cases.

\paragraph{Case 1.} $n\equiv 0~({\rm mod}~4)$.
Since $\alpha(G)\geq n/3\ge n/4$, we infer that $G$ is simply balanceable by the above observation.

\paragraph{Case 2.} $n\equiv 1~({\rm mod}~4)$.
Let $p$ be the positive integer such that $n=4p+1$.
Then $|E| = 4p+1$.
Since $\alpha(G)\geq n/3 = (4p+1)/3\ge p$, there exists an independent set $I$ in $G$ of size $p$.
The sum of the degrees of vertices in $I$ is \[\sum_{x\in I}d(x)=2|I|=2p = \left\lfloor\frac{|E|}{2}\right\rfloor\,.\]
Hence, $G$ is simply balanceable.

\paragraph{Case 3.} $n\equiv 3~({\rm mod}~4)$.
Let $p$ be the positive integer such that $n=4p+3$.
Then $|E| = 4p+3$.
Similarly as before, $\alpha(G)\geq n/3 = (4p+3)/3\ge p+1$; in particular, $G$ contains an independent set $I$ of size $p+1$.
The sum of the degrees of vertices in $I$ is \[\sum_{x\in I}d(x)=2|I|=2p+2 = \left\lceil\frac{|E|}{2}\right\rceil\,.\]
Hence, $G$ is simply balanceable.

\medskip
Next, we show that the conditions are necessary.
Assume that $G = (V,E)$ is simply balanceable and let $I$ be an independent set in $G$ such that
\[\left\lfloor\frac{|E|}{2}\right\rfloor\le \sum_{x \in I} d(x)\le \left\lceil\frac{|E|}{2}\right\rceil\,.\]
Since $\sum_{x \in I} d(x) = k|I|$ and $|E| = kn/2$, the above condition simplifies to
\[\left\lfloor\frac{kn}{4}\right\rfloor\le k|I|\le \left\lceil\frac{kn}{4}\right\rceil\,.\]
Suppose first that $|E|$ is even.
Then $kn/4=|E|/2$ is an integer, implying that $k|I| = kn/4$ and, hence, that $\alpha(G) \ge |I| = n/4$ and $n\equiv 0~({\rm mod}~4)$.
In both cases, when either $k = 2$ or $k\ge 3$, the corresponding condition from the theorem statement is satisfied.
We may thus assume that $|E|$ is odd.
Then $|E|/2 = kn/4$ is an integer multiple of $1/2$ and
\[kn/4-1/2\le\left\lfloor\frac{kn}{4}\right\rfloor\le k|I|\le \left\lceil\frac{kn}{4}\right\rceil\le kn/4+1/2\,.\]
Let $p = \left\lfloor\frac{n}{4}\right\rfloor$ and $q = n-4p$.
Then $q\in \{1,2,3\}$ and
\[n/4-1/(2k)\le |I|\le n/4+1/(2k)\,.\]
Using $n/4 = p+q/4$, we obtain
\begin{equation}\label{eq:0}
q/4-1/(2k)\le |I|-p\le q/4+1/(2k)\,.
\end{equation}
Assume that $k$ is odd.
Then $n$ is even and thus $q = 2$, and the chain of inequalities \eqref{eq:0} simplifies to \[1/2-1/(2k)\le |I|-p\le 1/2+1/(2k)\,,\] which together with $k\ge 3$ implies $1/3\le |I|-p\le 2/3$, a contradiction, since $|I|-p$ is an integer.
Thus, $k$ is even.
Since $|E| = kn/2$ is odd, this implies that $n$ is odd, that is, $q\in \{1,3\}$.
If $q = 1$, then \eqref{eq:0} implies \[0\le 1/4-1/(2k)\le |I|-p\le 1/4+1/(2k)\le 1/2\,.\]
Since $|I|-p$ is an integer, this is only possible if $|I| = p$ and $k = 2$.
Thus, $G$ is $2$-regular and since $n$ is odd, $n\not\equiv 2~({\rm mod}~4)$.

Finally, if $q = 3$, then \eqref{eq:0} implies \[1/2\le 3/4-1/(2k)\le |I|-p\le 3/4+1/(2k)\le 1\,.\]
Since $|I|-p$ is an integer, this is only possible if $|I| = p+1$ and $k = 2$.
Thus, $G$ is $2$-regular and since $n$ is odd, $n\not\equiv 2~({\rm mod}~4)$.
\end{proof}

We now summarize the main consequences of \Cref{lem:regular-sb} for simple balanceability of $k$-regular graphs for $k \le 4$.
Dailly, Hansberg, and Ventura~\cite{main} showed that for even $n\ge 4$, the $n$-vertex cycle is balanceable if and only if $n\equiv 0~({\rm mod}~4)$.
Using \Cref{lem:regular-sb}, this result can be generalized as follows.

\begin{sloppypar}
\begin{corollary}\label{thm:2-regular}
Let $G$ be a $2$-regular graph of order $n$.
Then, the following statements are equivalent.
\begin{enumerate}
    \item $G$ is balanceable.
    \item $G$ is simply balanceable.
    \item $n\not\equiv 2~({\rm mod}~4)$.
\end{enumerate}
\end{corollary}
\end{sloppypar}

\begin{proof}
Let $G=(V,E)$ be a $2$-regular graph and let $n$ be the order of $G$.
Since $G$ is $2$-regular, we have $2|V|=\sum_{x\in V}d(x) = 2|E|$ and hence $|E| = n$.
If $G$ is balanceable, then, by \Cref{prop:Eulerian}, $n = |E|$ is either odd or a multiple of $4$, that is, \hbox{$n\not\equiv 2~({\rm mod}~4)$}.
If \hbox{$n\not\equiv 2~({\rm mod}~4)$}, then $G$ is simply balanceable by \Cref{lem:regular-sb}.
Finally, if $G$ is simply balanceable, then $G$ is balanceable by \Cref{indep}.
\end{proof}

\begin{remark}\label{remark}
\Cref{thm:2-regular} shows, in particular, that the classes of balanceable and simply balanceable graphs are not closed under disjoint union: the cycle $C_3$ is simply balanceable, while the disjoint union of two copies of $C_3$ is not balanceable.
Similarly, a graph may be balanceable (or even simply balanceable) even if none of its components are balanceable.
This is the case for the graph consisting of the disjoint union of two copies of $C_6$.
\hfill $\blacktriangle$
\end{remark}

\begin{corollary}\label{lem:k23n0mod4}
Let $G$ be a cubic graph of order $n$.
Then $G$ is simply balanceable if and only if $n\equiv 0~({\rm mod}~4)$.
\end{corollary}

\begin{proof}
Let $G = (V,E)$ be a cubic graph of order $n$.
By \Cref{lem:regular-sb}, it suffices to show that $\alpha(G)\geq n/4$.
But this follows from the fact that $\chi(G)\leq \Delta(G)+1 \le 4$.
\end{proof}

\begin{sloppypar}
\begin{corollary}\label{lem:connected}
Let $G$ be a  connected $4$-regular graph of order $n$.
Then $G$ is simply balanceable if and only if \hbox{$n\equiv 0~({\rm mod}~4)$}.
\end{corollary}
\end{sloppypar}

\begin{proof}
Let $G = (V,E)$ be a connected $4$-regular graph of order $n$.
By \cref{lem:regular-sb}, it suffices to show that $\alpha(G)\ge n/4$.
Since $K_5$ is the only $4$-regular complete graph, $G$ is not complete.
Since $G$ is also not an odd cycle, Brooks' theorem implies that $\chi(G)\leq 4$ and hence, the independence number $\alpha(G)$ is at least $n/4$.
\end{proof}

Note that the connectedness assumption in \Cref{lem:connected} is necessary.
For example, let $G$ be the graph consisting of the disjoint union of four complete graphs $K_5$.
Then, $G$ is a $4$-regular graph of order \hbox{$n = 20$}.
Even though $n\equiv 0~({\rm mod}~4)$, we infer using \cref{lem:regular-sb} that $G$ is not simply balanceable, since the independence number of $G$ satisfies $\alpha(G) = 4<n/4$.

The above examples indicate that when studying balanceability and simple balanceability of graphs in a certain class, extending the result from the connected case to the general case may not always be straightforward.
While the extension was simple in the case of $2$-regular graphs and---as we will see in \cref{cubic}---is also simple in the $3$-regular case, the $4$-regular case turns out to be significantly more involved (see, e.g., the proof of~\Cref{Thm:4-regular}).

\medskip
The connection between simple balanceability of regular graphs and their independence numbers established in \cref{lem:regular-sb} suggests that the problem of recognizing simply balanceable graphs might be \NP-complete.
Before returning to the study of balanceable graphs, we prove in the next section that this is indeed the case.

\section{Recognizing simply balanceable graphs is \NP-complete}\label{sec:complexity}

Our \NP-completeness proof is based on a reduction from the following problem, proved \NP-complete by Leven and Galil~\cite{zbMATH03804843}: Given a $4$-regular graph $G$, is $G$ $4$-edge-colorable?

We will also use the following well-known fact relating $p$-colorability of a graph with the independence number of the Cartesian product of the graph with the complete graph $K_p$ (see, e.g.,~\cite[p.~380]{zbMATH03400923}).

\begin{proposition}\label{prop:p-colorable}
For any $n$-vertex graph $H$ and positive integer $p$, the graph $H$ is $p$-colorable if and only if $\alpha(H\Box K_p) = n$.
\end{proposition}

\begin{theorem}\label{NP-complete}
Determining whether a given $9$-regular graph is simply balanceable is \NP-complete.
\end{theorem}

\begin{proof}
Clearly, the problem belongs to \NP.
Given an independent set $I$ satisfying the condition from the definition of simply balanceable graphs, it can be verified in polynomial time whether $I$ is indeed an independent set satisfying the condition.

To show \NP-hardness, we make a reduction from the problem of $4$-edge-coloring $4$-regular graphs.
Given a $4$-regular graph $G$, compute the line graph $L(G)$ and let $G'$ be the Cartesian product of $L(G)$ with the complete graph $K_4$.
In other words, take four copies of $L(G)$, take their disjoint union, and for each vertex of $L(G)$, add edges between any two copies of the vertex.
Since each edge of $G$ shares an endpoint with precisely $6$ other edges, the line graph of $G$ is $6$-regular.
Consequently, each vertex of $G'$ is adjacent to $6$ other vertices in the corresponding copy of $L(G)$ and to $3$ other copies of the vertex in $G'$, making $G'$ a $9$-regular graph.
Let us denote by $m$ the number of edges of $G$.
Since the vertex set of $G'$ is partitioned into $|V(L(G))| = m$ many $4$-cliques, the independence number of $G'$ is at most $m$.
Furthermore, the number of vertices of $G'$ is equal to $4m$; in particular, $|V(G')|\equiv 0~({\rm mod}~4)$.
By \Cref{lem:regular-sb}, $G'$ is simply balanceable if and only if $\alpha(G')\ge m$.

To complete the proof, we show that $G$ is $4$-edge-colorable if and only if $\alpha(G')\ge m$.
Note first that $G$ is $4$-edge-colorable if and only if $L(G)$ is $4$-colorable.
Since $G' = L(G)\Box K_4$, it follows from \cref{prop:p-colorable} that $L(G)$ is $4$-colorable if and only if $\alpha(G') = |V(L(G))| = m$.
As we already observed that $\alpha(G')\le m$, we conclude that $G$ is $4$-edge-colorable if and only if $\alpha(G') \ge m$.
\end{proof}

\section{Balanceability of cubic graphs}\label{cubic}

As mentioned in the introduction, one of the aims of this paper is to study the balanceability of cubic graphs, addressing a question of Dailly, Hansberg, and Ventura~\cite{main}.
The following theorem confirms that cubic graphs are balanceable.

\begin{theorem}\label{Thm:cubic graph}
Every cubic graph is balanceable.
\end{theorem}

\begin{proof}
Let $G=(V,E)$ be a cubic graph and let $n$ be the order of $G$.
Since $G$ is cubic, $n$ is even.
We consider two cases depending on the value of $n\bmod 4$.

\paragraph{Case 1.} $n\equiv 0~({\rm mod}~4)$.
In this case, $G$ is balanceable by \Cref{indep,lem:k23n0mod4}.

\paragraph{Case 2:} $n\equiv 2~({\rm mod}~4)$.
Let $k$ be the positive integer such that $n=4k+2$.
Then $|E| = 6k+3$.
Our goal is to show that the conditions of \Cref{partition} are satisfied.
In fact, we will show that $G$ has a set of vertices $W \subseteq V$ and a partition $V=\{X,Y\}$ such that $e_G(W)=e_G(X, Y) = 3k+1$.

First, we show that there exists a partition $\{X,Y\}$ of $V$ such that $e_G(X,Y)=3k+1$.
We are going to construct $X$.
Fix an arbitrary edge $xy\in E$.
Let $J=\{x,y\}$ and $H=G-N[J]$.
Since the graph $G$ is cubic, $|N[J]|\le 6$ and hence, $|V(H)|\ge n-6$.
As $H$ is a subgraph of $G$, we have $\chi(H)\leq\chi(G)\leq 4$, which implies that $\alpha(H)\ge  |V(H)|/4\geq (n-6)/4=k-1$; in particular, there exists an independent set $I$ in $H$ such that $|I| = k-1$.
Let $X = I\cup J$ and $Y= V\setminus X$; see \Cref{Fig:cubic graph}.

\begin{figure}[h!]
    \centering
		\begin{tikzpicture}[scale=1.3]
\draw [rotate around={81.8698976458454:(-1.56,3.48)},line width=1.2pt] (-1.56,3.48) ellipse (1.82786947059178cm and 1.6189832616557454cm);
\draw [rotate around={81.86989764584541:(2.02,3.56)},line width=1.2pt] (2.02,3.56) ellipse (1.8278694705917768cm and 1.6189832616557427cm);
\draw [line width=2.pt,dotted] (-3.090847313658136,4.028220675271917)-- (0.,4.044208981538004);
\draw [line width=0.8pt] (-1.5,4.886666666666666)-- (2.,5.2);
\draw [line width=0.8pt] (-1.5,4.886666666666666)-- (2.,5.);
\draw [line width=0.8pt] (-1.5,4.5)-- (2.,4.8);
\draw [line width=0.8pt] (-1.5,4.5)-- (2.,4.6);
\draw [line width=0.8pt] (-1.5133333333333343,3.686666666666666)-- (2.001672135052165,3.9117106151847634);
\draw [line width=0.8pt] (-1.5133333333333343,3.686666666666666)-- (2.001672135052165,3.7881638705323435);
\draw [line width=0.8pt] (-1.5,3.3)-- (1.9963005374585814,3.2348893183932463);
\draw [line width=0.8pt] (-1.5,3.3)-- (1.9963005374585814,3.11671417133441);
\draw [line width=0.8pt] (-1.5,4.886666666666666)-- (-1.5,4.5);
\draw [line width=0.8pt] (-1.4866666666666677,2.073333333333333)-- (1.990928939864998,2.3485757154519744);
\draw [line width=0.8pt] (-1.4866666666666677,2.073333333333333)-- (1.9963005374585814,2.2196573732059712);
\draw [line width=0.8pt] (-1.5133333333333343,3.686666666666666)-- (2.001672135052165,3.664617125879924);
\draw [line width=0.8pt] (-1.5,3.3)-- (1.9963005374585814,3.353064465452083);
\draw [line width=0.8pt] (-1.4866666666666677,2.073333333333333)-- (2.001672135052165,2.090739030959968);
\begin{scriptsize}
\draw [fill=black] (-1.5,4.886666666666666) circle (1.3pt);
\draw[color=black] (-1.4442688823907668,5.041287969232268) node {$x$};
\draw [fill=black] (-1.5,4.5) circle (1.3pt);
\draw[color=black] (-1.4442688823907668,4.3) node {$y$};
\draw [fill=black] (-1.5133333333333343,3.686666666666666) circle (1.3pt);
\draw [fill=black] (-1.5,3.3) circle (1.3pt);
\draw[color=black] (-1.4271668350085396,2.8864299990716606) node {$\vdots$};
\draw [fill=black] (-1.4866666666666677,2.073333333333333) circle (1.3pt);
\draw [fill=black] (2.,5.2) circle (1.3pt);
\draw [fill=black] (2.,5.) circle (1.3pt);
\draw [fill=black] (2.,4.8) circle (1.3pt);
\draw [fill=black] (2.,4.6) circle (1.3pt);
\draw [fill=black] (2.001672135052165,3.9117106151847634) circle (1.3pt);
\draw [fill=black] (2.001672135052165,3.7881638705323435) circle (1.3pt);
\draw [fill=black] (1.9963005374585814,3.2348893183932463) circle (1.3pt);
\draw [fill=black] (1.9963005374585814,3.11671417133441) circle (1.3pt);
\draw [fill=black] (1.990928939864998,2.3485757154519744) circle (1.3pt);
\draw [fill=black] (1.9963005374585814,2.2196573732059712) circle (1.3pt);

\draw[color=black] (-1.4271668350085396,1.2959395925245456) node {$X$};
\draw[color=black] (2.087303902039137,1.2788375451423184) node {$V\setminus X$};
\draw [fill=black] (2.001672135052165,3.664617125879924) circle (1.3pt);
\draw [fill=black] (1.9963005374585814,3.353064465452083) circle (1.3pt);
\draw [fill=black] (2.001672135052165,2.090739030959968) circle (1.3pt);
\draw [decorate,decoration={brace,amplitude=10pt,mirror,raise=4pt},yshift=0pt] (-3.2,4) -- (-3.2,1.5)  node [black,midway,xshift=-0.9cm] {\footnotesize $k-1$};
\end{scriptsize}
\end{tikzpicture}
\caption{Proof of Case 2 of \Cref{Thm:cubic graph} when $|N(\{x,y\})|=4$.
The subgraph of $G$ induced by $X$ has a unique edge $xy$ and $e_G(X,V\setminus X)=3k+1$.\label{Fig:cubic graph}}
\end{figure}
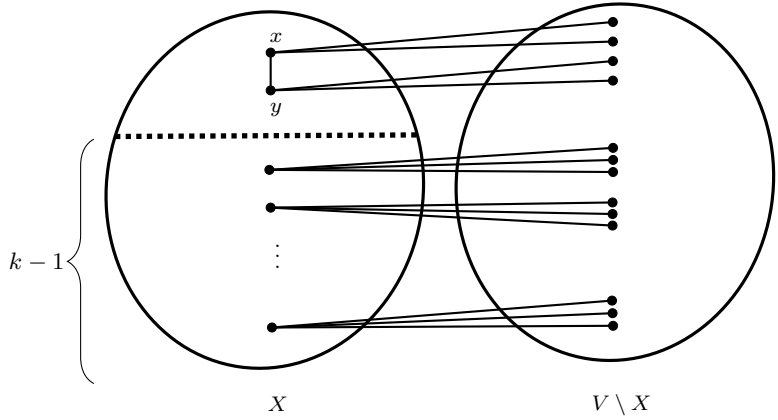

Note that $xy$ is the unique edge of the subgraph of $G$ induced by $X$.
Thus, using also the fact that $G$ is cubic, we obtain that \[e_G(X,Y) = e_G(I,Y)+ e_G(J,Y) = 3(k-1)+4 = 3k+1\,,\]
as claimed.

Second, to show that $G$ has a set of vertices $W \subseteq V$ such that $e_G(W)= 3k+1$, observe that the set $Y$ satisfies the desired condition.
Indeed,
\[e_G(Y) = |E|-e_G(X)-e_G(X,Y) = (6k+3)-1-(3k+1)=3k+1\,.\]
Applying \Cref{partition}, this completes the proof that $G$ is balanceable.
\end{proof}

\section{Balanceability of $4$-regular graphs}\label{4-regular}

\Cref{prop:Eulerian} implies the following.

\begin{corollary}\label{cor:4-regular}
No $4$-regular graph of odd order is balanceable.
\end{corollary}

\begin{proof}
Let $G = (V,E)$ be a $4$-regular graph.
Then $4|V| = 2|E|$ and hence, $|V| = |E|/2$; in particular, $|E|$ is even.
By~\Cref{prop:Eulerian}, $|E|$ is a multiple of $4$.
Consequently $|V|$ is even.
\end{proof}

By~\Cref{cor:4-regular}, a necessary condition for the balanceability of a $4$-regular graph is that it has even order.
In the case of connected $4$-regular graphs of order $n\equiv 2~({\rm mod}~4)$, some graphs are balanceable and some are not (see~\cite{main}).
However, we next show that in the case $n\equiv 0~({\rm mod}~4)$ they are all balanceable.
Note that \cref{indep,lem:connected} imply the following.

\begin{lemma}\label{lem:connected-balanceable}
    Every connected $4$-regular graph of order $n$, where $n\equiv 0~({\rm mod}~4)$, is balanceable.
\end{lemma}

For the general case, we will make use of the following consequence of \Cref{lem:join}.

\begin{corollary}\label{cor:join-4-regular}
Let $G$ and $H$ be balanceable graphs such that $H$ is $4$-regular.
Then $G+H$ is balanceable.
\end{corollary}

We will also need the following sufficient condition for balanceability of $4$-regular graphs.

\begin{lemma}\label{lem:k4n0mod4}
Every $4$-colorable $4$-regular graph of order $n$, where $n\equiv 0~({\rm mod}~4)$,  is balanceable.
\end{lemma}

We omit the proof of \cref{lem:k4n0mod4}, since it is the same as the proof of~\cref{lem:k23n0mod4}.

\begin{theorem} \label{Thm:4-regular}
Every $4$-regular graph of order $n$, where $n\equiv 0~({\rm mod}~4)$, is balanceable.
\end{theorem}

\begin{proof}
Suppose for a contradiction that there exists a $4$-regular graph of order $n$ divisible by $4$ that is not balanceable, and let $G$ be a counterexample minimizing $n$.
By \Cref{lem:connected-balanceable}, $G$ is disconnected.
In the proof, we will analyze the structure of components of $G$, eventually reaching a contradiction.
Let $\mathcal{K}$ be the set of components of $G$ that are complete graphs (that is, isomorphic to $K_5$), let $U$ be the set of vertices belonging to a component in $\mathcal{K}$, and let $a$ be the cardinality of $\mathcal{K}$.
Furthermore, for $i\in \{0,1,2,3\}$, let $a_i$ be the number of components $C$ of $G-U$ such that and $|V(C)| \equiv i~({\rm mod}~4)$.

We continue the proof by proving several claims.

\begin{claim}\label{claim0}
Let $\mathcal{C}$ be a nonempty proper subset of the set of components of $G$ and let $S$ by the set of vertices of $G$ that belong to a component in $\mathcal{C}$.
Then $|S|\not\equiv 0~({\rm mod}~4)$.
\end{claim}

\begin{proof}
Suppose that $|S|\equiv 0~({\rm mod}~4)$.
The assumption on $\mathcal{C}$ implies that the set $S$ is a nonempty proper subset of $V(G)$.
The subgraphs of $G$ induced by $S$ and $V(G)\setminus S$ are $4$-regular and of order divisible by $4$, hence, by the minimality of $G$, they are both balanceable.
By \Cref{cor:join-4-regular}, $G$ is balanceable, a contradiction.
\end{proof}

\begin{claim}
    $a_0=0$.
\end{claim}

\begin{proof}
Suppose that $a_0>0$.
Then, there exists a component $C$ of $G$ such that $|V(C)| \equiv 0\pmod 4$.
Since $G$ is disconnected, the set $\{C\}$ is a nonempty proper subset of the set of components of $G$.
Hence, $|V(C)| \not\equiv 0\pmod 4$ by~\Cref{claim0}, a contradiction.
\end{proof}

 \begin{claim}\label{claim5}
    $a\in\{1,2,3\}$.
\end{claim}

\begin{proof}
If $a = 0$, then by Brooks' theorem, $G$ is $4$-colorable and, hence, balanceable by \Cref{lem:k4n0mod4}, a contradiction.
Therefore, $a\ge 1$.
Suppose for a contradiction that $a\ge 4$.
Choose a set $\mathcal{C}$ of any $4$~components from $\mathcal{K}$ and let $S$ by the set of vertices of $G$ that belong to a component in $\mathcal{C}$.
If $S \neq V(G)$, then \cref{claim0} implies that $|S|\not\equiv 0~({\rm mod}~4)$, a contradiction.
Hence, $S = V(G)$.
In particular, $\mathcal{C} = \mathcal{K}$ and $G$ is the disjoint union of four copies of $K_5$.
Choose any two components of $G$, let $V_1$ and $V_2$ denote their vertex sets, and let $H$ be the subgraph of $G$ induced by $V_1\cup V_2$.
Fix an arbitrary vertex $v_1\in V_1$ and two distinct vertices $v_2, v_2'\in V_2$.
Then $H$ is balanceable, as verified by applying \Cref{partition} to the set $W = V_1$ and the partition $\{X,Y\}$ of $V(G)$ given by $X = \{v_1,v_2,v_2'\}$ and $Y = V(H)\setminus X$.
Note that $G$ is isomorphic to the disjoint union of two copies of $H$.
Since $|E(H)| = 20$, \cref{lem:join} implies that $G$ is balanceable, a contradiction.
Therefore, $a\le 3$ and we conclude that $a\in\{1,2,3\}$, as claimed.
\end{proof}

\begin{claim}\label{claima3}
    $a_3=0$.
\end{claim}
\begin{proof}
    Suppose that $a_3>0$.
    Let $H$ be a component of $G$ such that $|V(H)|=4k+3$ for some integer $k$.
    By \cref{claim5}, $G$ has a component $K$ that is  isomorphic to $K_5$.
    Note that $|V(K+H)|\equiv 0 \pmod 4$.
    If $V(K+H)\neq V(G)$, then~\cref{claim0} implies that \hbox{$|V(K+H)| \not\equiv 0\pmod 4$}, a contradiction.
    Hence $G = K+H$.
    By Brooks' theorem, $H$ has an independent set $I$ such that $|I|=\left\lceil \frac{|V(H)|}{4}\right\rceil=k+1$. Let $J=I\cup\{v\}$, where $v$ is an arbitrary vertex of $K$. Thus $J$ is an independent set of $G$ and $$\sum_{x\in J}d(x)=4|J|=4(k+2)=4k+8=|V(G)|=\frac{|E(G)|}{2}.$$
    So $G$ is balanceable by \Cref{indep}, a contradiction.
\end{proof}

\begin{claim}\label{claim3}
    $a\leq 1$ or $a_2=0$.
\end{claim}

\begin{proof}
    Suppose that $a\geq 2$ and $a_2>0$.
    Let $H$ be an induced subgraph of $G$ consisting of two components isomorphic to $K_5$ and let $H'$ be a component of $G$ with $|V(H')|\equiv 2 \pmod 4$.
    Note that $|V(H'+H)|\equiv 0 \pmod 4$.
    If $V(H'+H)\neq V(G)$, then~\cref{claim0} implies that \hbox{$|V(H'+H)| \not\equiv 0\pmod 4$}, a contradiction.
    Hence $G = H+H'$; in particular, $G$ is isomorphic to the disjoint union of $H'$ and two copies of $K_5$.
    We have $|V(H')|=4k+2$ for some integer $k$.
    Using Brooks' theorem, we infer that $\chi(H')\leq 4$.
    Thus, there exists an independent set $I'$ in $H'$ such that $|I'|=\left\lceil\frac{|V(H')|}{4}\right\rceil=k+1$. Choose two vertices $u$ and $v$ in different components of $H$ and set $I=I'\cup \{u,v\}$.
    Then $|I|=k+3$ and $|V(G)|=4k+2+10=4(k+3)=4|I|$.
    So $\sum_{x\in I}d(x)=4|I|=|V(G)|=\frac{|E(G)|}{2}$.
    By \Cref{indep}, $G$ is balanceable, a contradiction.
\end{proof}

\begin{claim}\label{claim6}
    $a=1$ or $a_2=0$.
\end{claim}
\begin{proof}
    Immediately  from Claims~\ref{claim5} and \ref{claim3}.
\end{proof}

\begin{claim}\label{claim10}
    $a_1\leq 2$.
\end{claim}

\begin{proof}
    Suppose that $a_1\geq 3$.
    Let $F$ be an induced subgraph of $G$ consisting of one component isomorphic to $K_5$ and three components $C_j$, where $j\in \{1,2,3\}$, of $G-U$ of order $|V(C_j)|=4k_j+1$ such that $k_j$ is an integer.
    Since $|V(F)|\equiv 0 \pmod 4$, \Cref{claim0} implies that $G = F$.
    By Brooks' theorem, each component $C_j$ contains an independent set $I_j$ of size $k_j+1$.
    Put $I=I_1\cup I_2\cup I_3\cup\{v\}$.
    Then \[\sum_{x\in I}d(x)=4|I|=4(k+2)=4k+8=|V(G)|=\frac{|E(G)|}{2}\,,\] hence, $G$ is balanceable by \cref{indep}, a contradiction.
\end{proof}

\begin{claim}\label{claim7}
    $a_2>0$.
\end{claim}
\begin{proof}
    Suppose, contrary to our claim, that $a_2=0$.
    Then according to assumption of the theorem \[|V(G)|=\sum\{|V(C)|\colon C \text{ is a component of}~G\}\] should be $0$ modulo $4$.
    Using~\Cref{claima3}, this implies that $a+a_1\equiv 0\pmod 4$.
  By \cref{claim5}, we have $a\in\{1,2,3\}$.
  Since $a+a_1\equiv 0\pmod 4$, we infer that $a_1\equiv c\pmod 4$ where $c =4-a$.

    Set $G' = G-U$.
    Note that $c\in \{1,2,3\}$ and $G'$ has at least $c$ components.
    Choose a set $\mathcal{C}$ of $c$ components of $G'$ and let $F$ be the subgraph of $G$ induced by $U\cup \bigcup_{C\in\mathcal{C}}V(C)$.
    Then $|V(F)|\equiv 0 \pmod 4$ and \Cref{claim0} implies that $G = F$.
    Choose an independent set $J\subseteq U$ containing one vertex from each component in $\mathcal{K}$.
    For each component $C\in \mathcal{C}$, we have $|V(C)|=4k_C+1$ for some integer $k_C$.
    Brooks' theorem implies that each component $C\in \mathcal{C}$ has an independent set $J_C$ of size $\left\lceil\frac{|V(C)|}{4}\right\rceil=k_C+1$.
    Fix a component $D\in \mathcal{C}$.
    For each component $C\in \mathcal{C}\setminus\{D\}$, set  $I_C$ to be an arbitrary subset of $J_C$ such that $|I_C| = k_C$, that is, $I_C$ is obtained from $J_C$ by deleting from it an arbitrary vertex.
    Set \[I=J\cup J_D\cup \bigcup_{C\in\mathcal{C}\setminus\{D\}}I_C.\]
   Then $I$ is an independent set in $G$ such that \[|I| = a+(k_D+1)+\sum_{C\in \mathcal{C}\setminus\{D\}}k_C =\sum_{C\in \mathcal{C}}k_C+a+1\,.\]
    Furthermore,
  \[|V(G)| = |V(F)| = 5a+\sum_{C\in \mathcal{C}}(4k_C+1) = 4\sum_{C\in \mathcal{C}}k_C+5a+c = 4\sum_{C\in \mathcal{C}}k_C+4a+4\,.\]
  Therefore,
   \[\sum_{x\in I}d(x)=4|I|=4\sum_{C\in \mathcal{C}}k_C+4a+4=|V(G)| = \frac{|E(G)|}{2}\,,\]
    hence, $G$ is balanceable by \Cref{partition}, a contradiction.
\end{proof}
\begin{claim}\label{claim9}
    $a=1$.
\end{claim}
\begin{proof}
    Immediately from Claims~\ref{claim6} and~\ref{claim7}.
\end{proof}

\begin{claim}
    $a_1=1$.
\end{claim}
\begin{proof}
    Suppose for a contradiction that $a_1\neq 1$.
  Then \cref{claim10} implies that $a_1\in \{0,2\}$.
  Using Claims~\ref{claima3} and~\ref{claim9}, we infer that $|V(G)|\equiv a_1+1 \pmod 4$.
  Hence, $|V(G)|$ is odd, a contradiction.
\end{proof}

Now we are ready to complete the proof of \Cref{Thm:4-regular}.
Let $F$ be an induced subgraph of $G$ consisting of the unique component of $G$ isomorphic to $K_5$, the unique component $C$ of $G-U$ with $|V(C)|=4k+1$ for some integer $k$, and exactly one component $C'$ with $|V(C')|=4k'+2$ for some integer $k'$.
Since $|V(F)|\equiv 0 \pmod 4$, \Cref{claim0} implies that $G = F$.
Using Brooks' theorem, we can see that $C$ contains an independent set $I$ of size $k+1$ and $C'$ has an independent set $I'$ of size $k'+1$.
Set $J=I\cup I'$.
Note that $|V(G)| = |V(F)|=4(k+k')+8$, moreover $|J|=k+k'+2$ and hence $\sum_{x\in J}d(x)=|V(G)|=\frac{|E(G)|}{2}$.
So $G$ is balanceable due to \cref{indep}, a contradiction. This completes the proof.
\end{proof}

As shown by the disjoint union of $4$ copies of the complete graph $K_5$ (see the paragraph after the proof of \Cref{lem:connected}), \Cref{Thm:4-regular} cannot be strengthened by replacing balanceability with simple balanceability.

\Cref{Thm:4-regular} generalizes Lemmas 15 and 16 in~\cite{main}, which establish that certain $4$-regular circulant graphs are balanceable.
The theorem also implies the following.

\begin{corollary}
For every cubic graph $G$, the graph $G\Box K_2$ is balanceable.
\end{corollary}

\begin{proof}
It can be easily seen that $G\Box K_2$ is a $4$-regular graph.
Furthermore, since $G$ is cubic, it has an even number of vertices.
Hence $n=|V(G\Box K_2)|=|V(G)|\times |V(K_2)|$ is even and $n\equiv 0~({\rm mod}~4)$.
By \Cref{Thm:4-regular}, $G\Box K_2$ is balanceable.
\end{proof}

\section{Conclusion}\label{sec:concl}

We proved several new results about regular balanceable graphs.
Important tools in our approach are Brooks' theorem and a sufficient condition for balanceability, which we call simple balanceability.
We characterized simply balanceable regular graphs via a condition involving the independence number of the graph, leading, on the one hand, to polynomially recognizable cases for small degrees and, on the other hand, to an \NP-completeness proof for the problem of recognizing simply balanceable $9$-regular graphs.

While balanceability coincides with simple balanceability for $k$-regular graphs with $k\in \{1,2\}$, this is not the case for $k\in \{3,4\}$.
For these cases we obtained either complete (for $k = 3$) or partial (for $k = 4$) characterizations of balanceability.
In particular, the results of \Cref{sec:regular,cubic} answer the question regarding the complexity of recognizing balanceable graphs when the input is restricted to $k$-regular graphs for $k\le 3$.
For all $k\ge 4$, the question remains open.

Our results are summarized in \cref{table:results-b,table:results-sb}.

\medskip

 \begin{table}[h!]
  \begin{center}
    \label{tab:table1}
    \begin{tabular}{|c||c|c|}
              \hline
     $k$ &  Characterizing conditions & Recognition complexity\\
        \hline
        \hline
       $1$ &  /  (\cref{obs:1-regular}) & $\PP$ (trivial) \\
       \hline
       $2$ &   $n\not\equiv 2~({\rm mod}~4)$ (\cref{thm:2-regular}) & $\PP$  (\cref{thm:2-regular})\\
       \hline
       $3$ &  / (\cref{Thm:cubic graph}) &  $\PP$ (trivial)\\
       \hline
       $4$ &  ? ($n\equiv 0~({\rm mod}~2)$ is necessary by \cref{cor:4-regular}, & ? (in $\PP$ if $n\not\equiv 2~({\rm mod}~4)$, \\
     &  $n\equiv 0~({\rm mod}~4)$ is sufficient, by \cref{Thm:4-regular}) & by \cref{cor:4-regular,Thm:4-regular})\\
       \hline  $\ge 5$ &   ? & ?\\
       \hline
    \end{tabular}
  \end{center}
  \caption{Summary of our results on {\bf  balanceability} of $k$-regular graphs of order $n$, for various small values of $k$. $\NPc$ stands for ``$\NP$-complete'', $\PP$ for ``solvable in polynomial time''.
 Question marks denote open questions.}\label{table:results-b}
\end{table}
 \begin{table}[h!]
  \begin{center}
    \label{tab:table1}
    \begin{adjustbox}{max width=\textwidth}
    \begin{tabular}{|c||c|c|}
          \hline
     $k$ & Characterizing conditions & Recognition complexity\\
        \hline
        \hline
       $1$ & /  (\cref{obs:1-regular}) & $\PP$ (trivial) \\
       \hline
       $2$ &  $n\not\equiv 2~({\rm mod}~4)$ (\cref{thm:2-regular}) & $\PP$  (\cref{thm:2-regular})\\
       \hline
       $3$ & $n\equiv 0~({\rm mod}~4)$ (\cref{lem:k23n0mod4}) & $\PP$   (\cref{lem:k23n0mod4})\\
       \hline
       $4$ & $n\equiv 0~({\rm mod}~4)$ (if $G$ is connected) (\cref{lem:connected})& ? (in $\PP$ if $G$ is connected, by \cref{lem:connected})\\
       \hline
       $\ge 5$ &  $n\equiv 0~({\rm mod}~4)$ and $\alpha(G)\ge n/4$ (\cref{{lem:regular-sb}}) & $\NPc$ for $k = 9$ (\cref{prop:p-colorable}),\\
       &&  open for all other values of $k$\\
       \hline
    \end{tabular}
    \end{adjustbox}
  \end{center}
  \caption{Summary of our results on {\bf simple balanceability} of $k$-regular graphs of order $n$, for various small values of $k$. $\NPc$ stands for ``$\NP$-complete'', $\PP$ for ``solvable in polynomial time''.
   The question mark denotes an open question.}\label{table:results-sb}
\end{table}

\begin{sloppypar}
Besides determining the complexity of recognizing balanceable graphs within the classes of \hbox{$k$-regular graphs} for $k\ge 4$, or even in general (which Dailly, Hansberg, and Ventura conjectured to be $\NP$-complete~\cite{main}), it would be interesting classify the boundary between easy and difficult cases for the recognition of simply balanceable $k$-regular graphs.
In particular, is there an integer $k_0$ such that recognizing simply balanceable $k$-regular graphs is polynomial-time solvable for all $k < k_0$ and $\NP$-complete for all $k\ge k_0$?
The results of this paper imply that if such a $k_0$ exists, then $4\le k_0\le 9$, unless $\PP = \NP$.
Understanding the balanceability and simple balanceability properties within hereditary graph classes (that is, graph classes closed under vertex deletion) is also a largely unexplored topic, which we leave for future research.
\end{sloppypar}

\subsection*{Acknowledgments}

This work is supported in part by the Slovenian Research and Innovation Agency (I0-0035, research program P1-0285 and research projects J1-3001, J1-3002, J1-3003, J1-4008, J1-4084, N1-0160, and N1-0208), and by the research program CogniCom (0013103) at the University of Primorska.

\end{document}